\documentclass[11pt]{amsart}
\usepackage[colorlinks, citecolor=blue, dvipdfm, pagebackref]{hyperref}

\usepackage{extarrows}

\newtheorem{theorem}{Theorem}[section]
\newtheorem{lemma}[theorem]{Lemma}

\theoremstyle{definition}
\newtheorem{definition}[theorem]{Definition}

\newtheorem{example}[theorem]{Example}

\newtheorem{proposition}[theorem]{Proposition}
\newtheorem{corollary}[theorem]{Corollary}
\newtheorem{remark}[theorem]{Remark}

\theoremstyle{remark}

\newcommand{\be}{\begin{equation}}
\newcommand{\ee}{\end{equation}}

\numberwithin{equation}{section}

\begin{document}

\title{On some applications of Gauduchon metrics}

\author{Ping Li}
\address{School of Mathematical Sciences, Tongji University, Shanghai 200092, China}

\email{pingli@tongji.edu.cn\\
pinglimath@gmail.com}
\thanks{The author was partially supported by the National
Natural Science Foundation of China (Grant No. 11722109).}




\begin{abstract}
We apply the existence and special properties of Gauduchon metrics to give several applications. The first one is concerned with the implications of algebro-geometric nature under the existence of a Hermitian metric with nonnegative holomorphic sectional curvature. The second one is to show the non-existence of holomorphic sections on Hermitian vector bundles under certain conditions. The third one is to give a restriction on the $\partial\bar{\partial}$-closedness of some real $(n-1,n-1)$-forms on compact complex manifolds.
\end{abstract}

\maketitle

\section{Introduction and main results}\label{sectionintroduction}
Throughout this article denote by $(M^{n},\omega)$ a compact connected complex manifold of complex dimension $n\geq2$ endowed with a Hermitian metric whose associated positive $(1,1)$-form is $\omega$. By abuse of notation, $\omega$ itself is also called the Hermitian metric. Denote by $$\{\omega\}:=\big\{e^u\omega~|~u\in C^{\infty}(M;\mathbb{R})\big\}$$
the set of the conformal class of $\omega$.

The Hermitian metric $\omega$ is called \emph{K\"{a}hler, balanced} or \emph{Gauduchon} if $d\omega=0$, $d\omega^{n-1}=0$ or  $\partial\bar{\partial}\omega^{n-1}=0$ respectively.
On general compact complex manifolds with $n\geq 2$ the former two metrics may not exist. For example, Calabi-Eckmann manifolds $S^{2p+1}\times S^{2q+1}$ ($p+q\geq1$) carry no K\"{a}hler or balanced metrics (\cite[\S 4]{Mi}). Nevertheless, a classical result of Gauduchon (\cite{Ga1}) states that every Hermitian metric is conformal to a Gauduchon metric, which is unique up to rescaling when $n\geq2$.

The aim of this article is to discuss some applications related to this existence result and some special properties of Gauduchon metrics. In what follows we shall describe our main results.

A compact complex manifold is called \emph{uniruled} if it can be covered by rational curves. $M$ is said to have \emph{negative} Kodaira dimension, denoted by $\kappa(M)<0$, if $H^0(M,mK_M)=0$ for any positive integer $m$, i.e., any positive tensor power of the canonical line bundle $K_M$ has no nontrivial holomorphic sections. It is known that for a projective manifold uniruledness implies negative Kodaira dimension (\cite[p. 94]{De}), which is conjectured to be true for compact K\"{a}hler or even general compact complex manifolds. One major open problem in the classification theory of projective manifolds is that the converse also should be true (\cite[Conjecture 0.1]{BDPP}).

The notions of uniruledness and Kodaira dimension are of algebro-geometric nature and so it is natural to find differential-geometric criterions to characterize them.
In the two influential problem lists, S.-T. Yau asked that (\cite[Problem 47]{Yau2}, \cite[Problem 67]{Yau3}), if a compact complex manifold admits a K\"{a}hler metric with positive holomorphic sectional curvature (``HSC'' for short), whether or not it is rationally connected or uniruled. Recently this was affirmatively confirmed by X.-K. Yang (\cite{Ya2}). The next natural question is whether this remains true for (non-K\"{a}hler) Hermitian metrics. Building on ideas of Balas (\cite{Ba1}, \cite{Ba2}) Yang treated this question in an earlier article \cite{Ya1} by showing that the quasi-positivity of holomorphic sectional curvature of a Hermitian metric implies negative Kodaira dimension (\cite[Thm 1.2]{Ya1}).

\emph{Our first main result} is the following Theorem \ref{mainresult0}, which extends \cite[Thm 1.2]{Ya1} to the nonnegative version.
\begin{theorem}\label{mainresult0}
Let $(M^n,\omega)$ be a compact Hermitian manifold with $HSC(\omega)\geq0$. Then the Kodaira dimension $\kappa(M)\leq0$, where the equality $\kappa(M)=0$ occurs if and only if $HSC(\omega)\equiv0$, $\omega$ is conformally balanced, and $K_M$ is a holomorphic torsion, i.e., $mK_M=\mathcal{O}_M$ for some $m\in\mathbb{Z}_{>0}$.
\end{theorem}
When $n=2$ or $3$, some results in Theorem \ref{mainresult0} were also obtained in \cite[Thm 1.2]{Ba2}.

Theorem \ref{mainresult0} immediately yields
\begin{corollary}~
\begin{enumerate}
\item
If $\kappa(M)\geq1$, $M$ cannot carry a Hermitian metric with nonnegative HSC.
\item
If $\kappa(M)\geq0$, any non-conformally balanced metric cannot admit a Hermitian metric with nonnegative HSC.
\end{enumerate}
\end{corollary}

Recall that a compact complex manifold is called \emph{Moishezon} (resp. \emph{in Fujiki's class $\mathcal{C}$}) if it is bimeromorphic to a projective (resp. compact K\"{a}hler) manifold. Combining the proof in Theorem \ref{mainresult0} with a recent result in \cite{CRS}, Yau's aforementioned question indeed holds true for some special compact (non-K\"{a}hler) Hermitian manifolds.
\begin{theorem}\label{mainresult1}
Suppose that $M$ is either an $n$-dimensional Moishezon manifold or a $3$-dimensional compact complex manifold in Fujiki's class $\mathcal{C}$. Then $M$ is uniruled provided one of the following two conditions can be satisfied.
\begin{enumerate}
\item
 $M$ admits a Hermitian metric $\omega$ with quasi-positive $HSC(\omega)$.
\item
$M$ admits a non-conformally balanced metric $\omega$ with nonnegative $HSC(\omega)$ and $\omega$ is not conformal to a balanced metric.
\end{enumerate}
\end{theorem}

In a classical work (\cite{Yau1}) Yau showed that the existence of a K\"{a}hler metric with positive total scalar curvature on $M$ implies $\kappa(M)<0$, and on a compact complex surface is equivalent to the uniruledness. Gauduchon proved, in another classical work (\cite{Ga2}), that the existence of a Gauduchon metric with positive total Chern scalar curvature implies $\kappa(M)<0$ and thus improved Yau's result. Using Boucksom, Demailly, P\u{a}un and Peternell's criterion for uniruled projective manifolds (\cite{BDPP}), Heier and Wong observed that (\cite{HW}) a projective manifold equipped with a K\"{a}hler metric with positive total scalar curvature is uniruled. Chiose, Rasdeaconu and Suvaina obtained in \cite{CRS} that a compact Moishezon manifold is uniruled if and only if it admits a Gauduchon metric with positive total Chern scalar curvature. Very recently Yang systematically investigated in \cite{Ya3} the relations among total Chern scalar curvature of Gauduchon metrics, Kodaira dimension and the pseudo-effectiveness of canonical line bundles.

\emph{Our second main result} is to extend Gauduchon's aforementioned result to the vector bundle version, which is related to a claim in Kobayashi's book (\cite[p. 57, Thm 1.30]{Ko}). Before stating the result, let us fix some more notation.

Let $(E^r,h)$ be a Hermitian holomorphic vector bundle of rank $r$ on ($M^n$,$\omega$). The starting point in \cite[Chapter 3]{Ko} entitled ``Vanishing Theorems" is that the quasi-negativity of the mean curvature form $K$ of $(E^r,h)$ (details on $K$ can be found in Section \ref{section2}) implies the nonexistence of nontrivial holomorphic sections on $E$ (\cite[p. 52]{Ko}). Let
\be\label{greatest eigenvalue}\gamma(x):=\text{the greatest eigenvalue of $K$ at $x$},\qquad x\in M.\ee
This $\gamma$ is in general a \emph{continuous} function and may not be smooth.

The following result states that, if the underlying metric is Gauduchon and the function $\gamma$ is smooth, the condition of $K$ being quasi-negative can be relaxed to the negativity of the total $\gamma$ on $M$.
\begin{theorem}\label{mainresult2}
Let $(E,h)$ be a Hermitian holomorphic vector bundle over $(M^n,\omega_0)$ and $\gamma$ as in (\ref{greatest eigenvalue}). If $\omega_0$ is Gauduchon, $\gamma$ smooth and
\be\label{integral condition}\int_M\gamma\cdot\omega^n_0<0,\ee
then $E$ admits no nontrivial holomorphic sections.
\end{theorem}

When applying Theorem \ref{mainresult2} to the line bundle $mK_M$ with the induced metric it turns out that $\gamma=-mS_{\omega_0}$ (see Example \ref{example of line bundle}), where $S_{\omega_0}$ is the Chern scalar curvature of $\omega_0$. Therefore we have the following consequence due to Gauduchon (\cite[p. 134]{Ga2}).

\begin{corollary}[Gauduchon]\label{coro of main result2}
If a compact complex manifold $M$ is endowed with a Gauduchon metric with positive total Chern scalar curvature, then $\kappa(M)<0.$
\end{corollary}

\begin{remark}
Theorem \ref{mainresult2} and Corollary \ref{coro of main result2} were claimed in \cite[p. 57, Thm (1.30), Coro. (1.33)]{Ko} \emph{without} the condition of the metric being Gauduchon. Note also that there is a typo where the symbol $``<"$ in \cite[Coro. (1.33)]{Ko} should be $``>"$. If they were true for \emph{any} Hermitian metric, then the positivity of the total Chern scalar curvature of \emph{any} Hermitian metric would imply the negative Kodaira dimension, which is clearly false. Indeed, using Sz\'{e}kelyhidi-Tosatti-Weinkove's recent solution to the Gauduchon conjecture on compact complex manifolds (\cite{STW}), Yang showed in \cite[Thm 1.7]{Ya3} that any compact complex manifold admits a Hermitian metric with positive total Chern scalar curvature. The mistake in the proof of \cite[Thm (1.30)]{Ko} lies in the claim in \cite[p. 57, (1.32)]{Ko}, which would be clear after the proof in Theorem \ref{mainresult2} (see Remark \ref{mistake remark}).
\end{remark}

Our third observation is motivated by the recent work in \cite{ACS2} on the so-called Chern-Einstein problems and can be stated as follows.
\begin{theorem}\label{mainresult3}
Let $\omega$ be a Hermitian metric on $M^n$ and $f\in C^{\infty}(M;\mathbb{R})$.
\begin{enumerate}
\item
Assume that $\partial\bar{\partial}(f\omega^{n-1})=0$. Then this $f$ has constant sign. If moreover $f$ is not identically zero, then $f$ is a constant if and only if $\omega$ is a Gauduchon metric.

\item
If $d(f\omega)=0$, then $f$ has constant sign. Namely, either $f\equiv0$ or $\pm f\omega$ is a K\"{a}hler metric.
\end{enumerate}
\end{theorem}

The Chern-Ricci form $\text{Ric}(\omega)$ of $(M^n,\omega)$ is defined to be
$$\text{Ric}(\omega):=-\sqrt{-1}
\partial\bar{\partial}\log\det(\omega^n),$$
which is a closed real $(1,1)$-form and represents the first Bott-Chern class up to a factor $2\pi$: $$[\text{Ric}(\omega)]=2\pi c_1^{BC}(M).$$

The following example, which is exactly \cite[Thm A]{ACS2}, illustrates an interesting application of this result.
\begin{example}
If the Hermitian metric $\omega$ satisfies the Einstein-type equation
\be\label{Einstein type}\text{Ric}(\omega)=\lambda\omega,\qquad\lambda\in C^{\infty}(M;\mathbb{R}),\ee
the closedness of $\text{Ric}(\omega)$ and Theorem \ref{mainresult3} imply that either $\lambda\equiv0$ or $\pm\lambda\omega$ is a K\"{a}hler metric. This in particular yields that $c_1^{BC}(M)$ is definite. So when the factor $\lambda$ is not identically zero, the solution in (\ref{Einstein type}) is necessarily reduced to the classical K\"{a}hler-Einstein case. This is exactly \cite[Thm A]{ACS2}.
\end{example}

The rest of this article is organized as follows. After collecting some preliminaries in Section \ref{section2},  we recall in Section \ref{section3} some basic facts on Gauduchon metrics and prove two important lemmas related to them (Lemmas \ref{analogue of laplacian} and \ref{scalar curvature}). The proofs of Theorems \ref{mainresult0}, \ref{mainresult1}, \ref{mainresult2} and \ref{mainresult3} are presented respectively in Sections \ref{section4}, \ref{section5} and \ref{section6}.

\section*{Acknowledgements}
The author thanks Xiaokui Yang for his useful comments and suggestions, which enhance the quality of this article.

\section{Preliminary materials}\label{section2}
We briefly collect in this section some basic facts on Hermitian holomorphic vector bundles and Hermitian manifolds in the form we shall use them in this article. A thorough treatment can be found in \cite{Ko}.

Let $(E^r,h)$ be a Hermitian holomorphic vector bundle of rank $r$ on ($M^n$,$\omega$)
with Chern connection $\nabla$ and curvature tensor $$R:=\nabla^2\in\Gamma(\wedge^{1,1}T^{\ast}M\otimes E^{\ast}\otimes E).$$

Under a local frame field $\{s_1,\ldots,s_r\}$ of $E$ and local coordinates $\{z^1,\ldots,z^n\}$ on $M$, the curvature tensor $R$ and the Hermitian metrics $h$ and $\omega$ can be written locally as
\begin{eqnarray}\label{curvature tensor}
\left\{ \begin{array}{ll}
R=R^{\beta}_{i\bar{j}\alpha}dz^i\wedge d\bar{z}^j\otimes s^{\ast}_{\alpha}\otimes s_{\beta},\\
~\\
h=(h_{\alpha\bar{\beta}}):=\big(h(s_{\alpha},s_{\beta})\big),\\
~\\
\omega=\sqrt{-1}g_{i\bar{j}}dz^i\wedge d\bar{z}^j.\\
\end{array} \right.
\end{eqnarray}
Here and in what follows we always adopt the Einstein summation convention.

Let a Hermitian matrix $K$ be $$K:=(K_{\alpha\bar{\beta}}):=
(h_{\gamma\bar{\beta}}\cdot g^{i\bar{j}}\cdot
R^{\gamma}_{i\bar{j}\alpha}),\qquad \Big((g^{i\bar{j}}):=(g_{i\bar{j}})^{-1}\Big)$$
which defines a Hermitian form on the smooth sections of $E$ by
$$K(\xi,\eta):=
K_{\alpha\bar{\beta}}\xi^{\alpha}\bar{\eta^{\beta}},\qquad
\xi=\xi^{\alpha}s_{\alpha},~\eta=\eta^{\beta}s_{\beta}.$$

This $K$ is independent of the choices of $\{s_{\alpha}\}$ and $\{z^i\}$ and called the \emph{mean curvature form} of $E$ in the notation of \cite{Ko}.

For a conformal change of the metric $h$:
$$\widetilde{h}=e^uh,\qquad u\in C^{\infty}(M;\mathbb{R}),$$
direct calculations (\cite[p. 57]{Ko}) show that the mean curvature form $\widetilde{K}=(\widetilde{K}_{\alpha\bar{\beta}})$ of $\widetilde{h}$ transforms in the following manner
\be\label{transform of mean curvature form}
(\widetilde{K}_{\alpha\bar{\beta}})=e^u\Big[
\big(K_{\alpha\bar{\beta}}\big)+
(\Delta_{c,\omega}u)\big(h_{\alpha\bar{\beta}}\big)\Big].\ee
Here $\Delta_{c,\omega}(\cdot)$ is the complex Laplacian acting on smooth functions defined by
\be\label{complex laplacian}\Delta_{c,\omega}u:=-\text{tr}_{\omega}
\sqrt{-1}\partial\bar{\partial}u=-g^{i\bar{j}}
\frac{\partial^2u}{\partial z^i\partial\bar{z}^j}.\ee

Let $(M^n,\omega)$ be a compact Hermitian manifold with $$\omega=\sqrt{-1}g_{i\bar{j}}dz^i\wedge d\bar{z}^j$$ under the local coordinates $\{z_1,\ldots,z_n\}$. Following the notation in (\ref{curvature tensor}), the components of the curvature tensor $R$ of the Chern connection on the holomorphic tangent bundle $(T^{1,0}M,\omega)$ are given by
$$R_{i\bar{j}k\bar{l}}:=g_{p\bar{l}}R_{i\bar{j}k}^p=
-\frac{\partial^2g_{k\bar{l}}}{\partial z^i\partial\bar{z}^j}+g^{p\bar{q}}\frac{\partial g_{k\bar{q}}}{\partial z^i}\frac{\partial g_{p\bar{l}}}{\partial \bar{z}^j}.$$

For $p\in M$ and $v=v^i\frac{\partial}{\partial z^i}\in T_p^{1,0}M$,
the \emph{holomorphic sectional curvature} $H$ of $\omega$ \big(HSC($\omega$) for short\big) at the point $p$ and the direction $v$ is defined by
\be\label{hsc}H_p(v):=R(v,\bar{v},v,\bar{v})\big|_p:=R_{i\bar{j}k\bar{l}}\big|_p\cdot v^i\bar{v}^jv^k\bar{v}^l.\ee

$H$ is called \emph{nonnegative} if $H_p(v)\geq0$ for any pair $(p,v)$. $H$ is called \emph{quasi-positive} if it is nonnegative and $H_p(v)>0$ for some pair $(p,v)$.

The \emph{Chern scalar curvature} $S_{\omega}$ of $\omega$ is defined by
\be\label{Chern scalar curvature}S_{\omega}:=g^{i\bar{j}}g^{k\bar{l}}R_{i\bar{j}k\bar{l}},\ee
and similarly we define another scalar function $\widehat{S}_{\omega}$ by
\be\label{second scalar curvature}\widehat{S}_{\omega}:=g^{i\bar{l}}g^{k\bar{j}}R_{i\bar{j}k\bar{l}}.\ee

It is well-known that $S_{\omega}=\widehat{S}_{\omega}$ when $\omega$ is K\"{a}hler, which is (half of) the Riemann scalar curvature. But for general Hermitian metrics they may be different.

We end this section with two related examples, which shall be used in the sequel.
\begin{example}
The behavior of Chern scalar curvatures under a conformal change is as follows.
\be\label{scalar curvature change}
S_{\widetilde{\omega}}=e^{-u}(n\Delta_{c,\omega}u+S_{\omega}),\qquad \widetilde{\omega}:=e^u\omega.
\ee
\begin{proof}
\be\begin{split}
S_{\widetilde{\omega}}=\text{tr}_{\widetilde{\omega}}
\text{Ric}(\widetilde{\omega})&=-\widetilde{g}^{i\bar{j}}
\partial_i\partial_{\bar{j}}\log(\widetilde{\omega}^n)\\
&=-e^{-u}g^{i\bar{j}}\partial_i\partial_{\bar{j}}
\big[nu+\log(\omega^n)\big]\\
&=e^{-u}(n\Delta_{c,\omega}u+S_{\omega}).
\end{split}\nonumber\ee
\end{proof}
\end{example}

\begin{example}\label{example of line bundle}
Given $(M^n,\omega)$ and consider the line bundle $mK_M$ with the induced metric. By the definition of $\gamma$ in (\ref{greatest eigenvalue}) we have
$$\gamma=-mg^{i\bar{j}}R^{k}_{i\bar{j}k}=-mS_{\omega},$$
from which, together with Theorem \ref{mainresult2}, Corollary \ref{coro of main result2} follows. Applying a Weitzenb\"{o}ck's formula (\cite[p.51, Prop. (1.8)]{Ko}) to this situation yields
\be\label{weitzenbock formula}
-\Delta_{c,\omega}(\big|\sigma\big|^2_{\omega})
=\big|\nabla\sigma\big|^2_{\omega}+
mS_{\omega}\big|\sigma\big|^2_{\omega},\qquad\forall~\sigma\in H^0(M;mK_M),\ee
where $\big|\cdot\big|_{\omega}$ is the pointwise norm on $mK_M$ induced from $\omega$ and $\nabla$ the Chern connection on $mK_M$.
\end{example}

\section{Some properties of Gauduchon metrics}\label{section3}
The \emph{torsion $1$-form} $\theta$ of a Hermitian metric $\omega$ on $M^{n}$ is characterized by $$d\omega^{n-1}=\omega^{n-1}\wedge\theta$$
as the following map is an isomorphism:
$$\omega^{n-1}\wedge(\cdot):\qquad\Omega^1(M)
\stackrel{\cong}{\longrightarrow}\Omega^{2n-1}(M).$$

Clearly $\omega$ is balanced if and only if $\theta=0$. It also turns out that the condition of $\omega$ being Gauduchon can be rephrased as $d^{\ast}_{\omega}\theta=0$. Namely, $\omega$ is Gauduchon if and only if $\theta$ is coclosed with respect to $\omega$.

The usual Riemann Laplacian $\Delta_{\omega}$ is defined by
$$\Delta_{\omega}(\cdot):=d^{\ast}_{\omega}d(\cdot):\qquad C^{\infty}(M;\mathbb{R})\longrightarrow C^{\infty}(M;\mathbb{R}).$$

The two Laplacians $\Delta_{c,\omega}$ and $\Delta_{\omega}$ are related by the following
\be\label{comparison}
2\Delta_{c,\omega}u=\Delta_{\omega}u+<du,\theta>_{\omega}, \qquad~\forall~u\in C^{\infty}(M;\mathbb{R}),\ee
where $<\cdot,\cdot>_{\omega}$ is the pointwise inner product with respect to $\omega$. The equality (\ref{comparison}) is due to Gauduchon (\cite[p. 502]{Ga3}), and a detailed proof can be found in \cite[Appedix A]{ACS} or \cite[Lemma 3.2]{To}.

The original treatment of Gauduchon metrics by Gauduchon is in terms of the kernels of $\Delta_{c,\omega}$ and its formal adjoint $\Delta_{c,\omega}^{\ast}$ with respect to $\omega$. We summarize several related basic properties in the following proposition. More details can be found in \cite{Ga1}, \cite[\S 8]{Ga2} and \cite[p. 224]{LT}.
\begin{proposition}[Gauduchon]\label{Gauduchon properties}
Suppose that $(M^n,\omega)$ is a compact Hermitian manifold with $n\geq2$. Then
\begin{enumerate}
\item
$\text{dim}_{\mathbb{R}}\text{Ker}(\Delta_{c,\omega}^{\ast})=1$ and any $f\in\text{Ker}(\Delta_{c,\omega}^{\ast})$ has constant sign. This implies that there exists a \emph{unique positive} smooth function $f_0=f_0(\omega)\in\text{Ker}(\Delta_{c,\omega}^{\ast})$ such that
\begin{eqnarray}\label{f_0}
\left\{ \begin{array}{ll}
\Delta_{c,\omega}^{\ast}(f_0)=0\\
~\\
\int_Mf_0\omega^n=\int_M\omega^n.\\
\end{array} \right.
\end{eqnarray}
The metric $\omega$ is Gauduchon if and only if $f_0(\omega)\equiv1$. 

\item
For every Hermitian metric $\omega$, the metric $f_0^{\frac{1}{n-1}}\cdot\omega$ is Gauduchon. Moreover, $f_0(\lambda\omega)=f_0(\omega)$ for any $\lambda\in\mathbb{R}_{>0}$ and thus every conformal class contains a unique Gauduchon metric up to rescaling.
\end{enumerate}
\end{proposition}
An immediate consequence of (\ref{comparison}) is that 
the two Laplacians are the same if and only if $\theta=0$, i.e., $\omega$ is balanced. So for general Gauduchon metrics they may be different. The following lemma, which is a key ingredient in the proof of Theorem \ref{mainresult2}, says that for Gauduchon metrics $\Delta_{c,\omega}(\cdot)$ still behaves like $\Delta_{\omega}$.
\begin{lemma}[Gauduchon]\label{analogue of laplacian}
\begin{enumerate}
\item
For a Hermitian metric $\omega$ on $M$, we have
\be\label{Gauduchon integrate 0}\text{$\omega$ is Gauduchon $\Longleftrightarrow$ $\int_M(\Delta_{c,\omega}u)\omega^n=0,$~$\forall$ $u\in C^{\infty}(M;\mathbb{R})$}.\ee

\item
Let $\omega_0$ be a Gauduchon metric on $M$ and given  $f\in C^{\infty}(M;\mathbb{R})$. The equation \be\label{solution}\Delta_{c,\omega_0}u=f\ee has a solution $u\in C^{\infty}(M;\mathbb{R})$ if and only if $\int_Mf\omega^n_0=0.$ Moreover, in this case the solution $u$ is unique up to an additive constant.
\end{enumerate}
\end{lemma}
\begin{proof}
By integrating over $M$ on both sides of (\ref{comparison}) we see that
$$2\int_M(\Delta_{c,\omega}u)\omega^n=
\int_M<u,d^{\ast}_{\omega}\theta>_{\omega}
\omega^n,\qquad\forall~ u,$$
from which (\ref{Gauduchon integrate 0}) follows.

For part $(2)$,
the necessarity follows from (\ref{Gauduchon integrate 0}). For the sufficiency,
Hodge theory says that we have for each $\omega$
\be\label{hodgedecomposition}C^{\infty}(M;\mathbb{R})
=\text{Im}(\Delta_{c,\omega})
\oplus\text{Ker}(\Delta^{\ast}_{c,\omega}).\ee
In our case the metric $\omega_0$ is Gauduchon and so
Proposition \ref{Gauduchon properties} implies that $\text{Ker}(\Delta^{\ast}_{c,\omega_0})=\mathbb{R}$. Thus (\ref{hodgedecomposition}) reduces to
\be C^{\infty}(M;\mathbb{R})
=\text{Im}(\Delta_{c,\omega_0})
\oplus\mathbb{R}.\nonumber\ee
This yields the sufficient part.

For the uniqueness of $u$, we only note that $\Delta_{c,\omega}u=0$ implies the constancy of $u$ due to the maximum principle.
\end{proof}

\begin{remark}
Although this lemma is not explicitly stated in \cite{Ga1}, the materials for the proof are all contained there,  as we have seen. A sketchy proof of this result in the more generally almost-complex case is outlined in \cite[Thm 2.2]{CTW}.
\end{remark}
\begin{definition}
Let $\omega_0$ be the Gauduchon metric in $\{\omega\}$. The sign of the total Chern scalar curvature of $\omega_0$,
$\int_MS_{\omega_0}\omega^n_0$, is called the \emph{Gauduchon sign} of $\{\omega\}$. Due to the uniqueness of Gaucuchon metrics up to rescaling the Gauduchon sign is well-defined.
\end{definition}
With this notion understood, Lemma \ref{analogue of laplacian} yields the following
\begin{lemma}\label{scalar curvature}
In every conformal class $\{\omega\}$, there always exists a Hermitian metric $\widetilde{\omega}$ whose Chern scalar curvature $S_{\widetilde{\omega}}$ has constant sign, which is necessarily the same as the Gauduchon sign of $\{\omega\}$.
\end{lemma}
\begin{proof}
Let $\omega_0$ be the Gauduchon metric in $\{\omega\}$ and
\be\label{1}f:=-\frac{S_{\omega_0}}{n}+\frac{\int_MS_{\omega_0}\omega^n_0}{n\int_M\omega^n_0}.\ee

Note that
$$\int_Mf\omega^n_0=0.$$
So Lemma \ref{analogue of laplacian} implies that there exists a $u\in C^{\infty}(M;\mathbb{R})$ such that $\Delta_{c,\omega_0}u=f.$

Take $\widetilde{\omega}:=e^u\omega_0$. By (\ref{scalar curvature change}) we have
\be\begin{split}
S_{\widetilde{\omega}}&=e^{-u}
(n\Delta_{c,\omega_0}u+S_{\omega_0})\\
&=e^{-u}(nf+S_{\omega_0})\\
&=\frac{\int_MS_{\omega_0}\omega^n_0}{\int_M\omega^n_0}\cdot e^{-u},\qquad\big(\text{by (\ref{1})}\big)
\end{split}\nonumber\ee
which is the desired Hermitian metric.

For the necessarity, note that if $S_{\omega}$ has constant sign and $\omega_{0}:=f_0^{\frac{1}{n-1}}\omega$ is Gauduchon, then we have \big(cf.(\ref{total scalar curvature relation})\big)
$$\int_MS_{\omega_0}\omega^n_0=\int_Mf_0S_{\omega}\omega^n,$$
which has the same sign as that of $S_{\omega}$.
\end{proof}
\begin{remark}
When the Gauduchon sign is positive, Lemma \ref{scalar curvature} was treated in \cite[Thm 1.3]{Ya3}.
When the Gauduchon sign is negative, the metric $\widetilde{\omega}$ in Lemma \ref{scalar curvature} can even be chosen so that $S_{\widetilde{\omega}}$ is a (negative) constant, which is the main result in \cite[Thm 4.1]{ACS}. Note that when the Gauduchon sign is zero, \cite[Thm 3.1]{ACS} is included in Lemma \ref{scalar curvature}. It is conjectured there that this remains true when the Gauduchon sign is positive. These results as well as the conjecture can be viewed as the complex analogue of the classical Yamabe problem.
\end{remark}

\section{Proofs of Theorems \ref{mainresult0} and \ref{mainresult1}}\label{section4}
Let $(M^n,\omega)$ be a compact Hermitian manifold, and choose for each $p\in M$ a unitary basis $\{e_1,\ldots,e_n\}$ of $T^{1,0}_pM$. The proof of Theorem \ref{mainresult0} as well as Theorem \ref{mainresult1} depends on the following two lemmas.

\begin{lemma}\label{lemma1}
The nonnegativity (resp. quasi-positivity) of HSC($\omega$) implies that of $S_\omega+\widehat{S}_{\omega}$.
\end{lemma}
\begin{proof}
We apply a classical trick usually attributed to Berger to average HSC($\omega$) of unit lengths at $p$, which was first used to show that the sign of HSC of a \emph{K\"{a}hler} metric determines that of scalar curvature.
\be\label{average}
\begin{split}
&\int_{v\in T^{1,0}_pM,|v|=1}H_p(v)d\theta(v)
\qquad\big(\text{$d\theta(v)$: spherical measure on $\mathbb{S}^{2n-1}$}\big)\\
=&\int_{v\in T^{1,0}_pM,|v|=1}R(e_i,\overline{e_j},e_k,\overline{e_l})
v^i\overline{v^j}v^k\overline{v^l}d\theta(v)\qquad(v=\sum v^ie_i)\\
=&R(e_i,\overline{e_j},e_k,\overline{e_l})\cdot
\frac{\delta_{ij}\delta_{kl}+\delta_{il}\delta_{kj}}{n(n+1)}
\cdot\text{Vol}(\mathbb{S}^{2n-1})\\
=&\frac{S_{\omega}(p)+\widehat{S}_{\omega}(p)}{n(n+1)}\cdot
\text{Vol}(\mathbb{S}^{2n-1}),\qquad\big(\text{by (\ref{Chern scalar curvature}),~(\ref{second scalar curvature})}\big)
\end{split}
\ee
where the second equality is due to the classical identity
$$\frac{1}{\text{Vol}(\mathbb{S}^{2n-1})}
\int_{\mathbb{S}^{2n-1}}v^i\overline{v^j}v^k\overline{v^l}d\theta(v)
=\frac{\delta_{ij}\delta_{kl}+\delta_{il}\delta_{kj}}{n(n+1)}.$$
\end{proof}

\begin{lemma}\label{lemma2}
The nonnegativity (resp. quasi-positivity) of $S_{\omega}+\widehat{S}_{\omega}$ implies that of the Gauduchon sign of $\omega$.
\end{lemma}
\begin{proof}
Let $\omega_0:=f_0^{\frac{1}{n-1}}\omega$ be the Gauduchon metric in $\{\omega\}$ and $\theta_0$ its torsion $1$-form. The two total scalar curvatures $S_{(\cdot)}$ and $\widehat{S}_{(\cdot)}$ of $\omega_0$ and $\omega$ are related by (\cite[(1.7)]{Ba1})
\begin{eqnarray}\label{total scalar curvature relation}
\left\{ \begin{array}{ll}
\int_MS_{\omega_0}\omega_0^n=
\int_Mf_0S_{\omega}\omega^n\\
~\\
\int_M\widehat{S}_{\omega_0}\omega_0^n=\int_Mf_0\widehat{S}_{\omega}\omega^n,\\
\end{array} \right.
\end{eqnarray}
and moreover (\cite[p. 501, Coro. 2]{Ga3})
\be\label{inequality}
\int_M(S_{\omega_0}-\widehat{S}_{\omega_0})\omega^n_0
=\frac{1}{2}\int_M|\theta_0|^2\omega^n_0.\ee
Therefore
\be\label{hsc implies gauduchon}
\begin{split}
\int_MS_{\omega_0}\omega^n_0&=
\frac{1}{2}\int_M(S_{\omega_0}+\widehat{S}_{\omega_0})\omega^n_0
+\frac{1}{2}\int_M(S_{\omega_0}
-\widehat{S}_{\omega_0})\omega^n_0
\\
&=\frac{1}{2}\int_Mf_0(S_{\omega}+\widehat{S}_{\omega})\omega^n
+\frac{1}{4}\int_M|\theta_0|^2\omega^n_0,\qquad\big(\text{by (\ref{total scalar curvature relation}), (\ref{inequality})}\big)
\end{split}\ee
from which as well as the positivity of $f_0$ Lemma \ref{lemma2} follows.
\end{proof}
We are now ready to prove Theorem \ref{mainresult0}.
\begin{proof}
Still denote by $\omega_0$ the Gauduchon metric in $\{\omega\}$. Let $\widetilde{\omega}\in\{\omega\}$ be the chosen metric as in Lemma \ref{scalar curvature} so that the Chern scalar curvature $S_{\widetilde{\omega}}$ has constant sign, which is the same as that of $\int_MS_{\omega_0}\omega^n_0.$ The assumption $HSC(\omega)\geq0$ in Theorem \ref{mainresult0} implies from Lemmas \ref{lemma1} and \ref{lemma2} that $$\int_MS_{\omega_0}\omega^n_0\geq0.$$

\emph{Case $1$: $\int_MS_{\omega_0}\omega^n_0>0$.}

$S_{\widetilde{\omega}}$ is positive everywhere on $M$. Apply this $\widetilde{\omega}$ to (\ref{weitzenbock formula}) we deduce from the maximum principle of the complex Laplacian $\Delta_{c,\widetilde{\omega}}$ that $mK_M$ has no nontrivial holomorphic sections when $m$ is positive. Namely, the Kodaira dimension $\kappa(M)<0$. This in fact gives a direct proof of Corollary \ref{coro of main result2}.

\emph{Case $2$: $\int_MS_{\omega_0}\omega^n_0=0$.}

$S_{\widetilde{\omega}}\equiv0$. In this case (\ref{weitzenbock formula}) yields that
any holomorphic section $\sigma$ on $mK_M$ is $\nabla$-parallel. So either $\sigma\equiv0$ or the zero set $\text{zero}(\sigma)=\emptyset$.
If for some $m$ there exists a holomorphic section $\sigma_0$ on $mK_M$ with $\text{zero}(\sigma_0)=\emptyset$, then for any holomorphic section $\sigma$ on $mK_M$
the ratio $\frac{\sigma}{\sigma_0}$ is a well-defined holomorphic function on $M$, thus a constant. This means $$\text{dim}_{\mathbb{C}}H^0(M;mK_M)\leq1,\qquad\forall~m,$$ and so $\kappa(M)\leq0$.

We now characterize $\kappa(M)=0$. From the proof above the case $\kappa(M)=0$ occurs if and only if
\begin{eqnarray}\label{equivalent conditions}
\left\{ \begin{array}{ll}
\int_MS_{\omega_0}\omega^n_0=0,\\
~\\
\text{$H^0(M;mK_M)\neq0$ for some $m$}.\\
\end{array} \right.
\end{eqnarray}

By (\ref{hsc implies gauduchon}) $\int_MS_{\omega_0}\omega^n_0=0$ is equivalent to $\omega_0$ is balanced, i.e., $\omega$ is conformally balanced, and $S_{\omega}+\widehat{S}_{\omega}\equiv0$. By (\ref{average}) $S_{\omega}+\widehat{S}_{\omega}\equiv0$ is in turn equivalent to $HSC(\omega)\equiv0$.

It suffices to show that the condition of $H^0(M;mK_M)\neq0$ for some $m$ under our situation implies that $K_M$ is a holomorphic torsion. Indeed, $\kappa(M)=0$ rules out the existence of a conformal class with positive Gauduchon sign due to Case $1$. Combining this with $\int_MS_{\omega_0}\omega^n_0=0$ implies that the first Bott-Chern class $c_1^{BC}(M)=0$ (\cite[Thm 1.1]{Ya3}). This, together with the fact $H^0(M;mK_M)\neq0$ for some $m$, yields that $K_M$ is a holomormphic torsion (\cite[Thm 1.4]{To15}).

In summary, in our situation (\ref{equivalent conditions}) is equivalent to $HSC(\omega)\equiv0$, $\omega$ is conformally balanced, and $K_M$ is a holomorphic torsion. This completes the proof of Theorem \ref{mainresult0}.
\end{proof}

Next we shall explain that why the proof above leads to Theorem \ref{mainresult1}. In fact by (\ref{hsc implies gauduchon}) any of the two conditions in Theorem \ref{mainresult1} implies that the Gauduchon sign of $\{\omega\}$ is positive, which for the manifolds under consideration is equivalent to the uniruledness (\cite[Thm D]{CRS}).

Note that \cite[Thm D]{CRS} is based on some deep results in birational geometry. Nevertheless, for our purpose only one direction is needed, whose proof is a direct application of some classical results. So we sketch the proof here for the reader's convenience. Indeed, for \emph{any} compact complex manifold $M$ the existence of a conformal class with positive Gauduchon sign is equivalent to the non-pseudo-effectiveness of $K_M$ (\cite[Thm 1.2]{La}, \cite[Thm 2.3]{Ya3}). Thus it suffices to show that the latter condition implies uniruledness for manifolds in question. Since both uniruledness and pseudo-effectiveness are bimeromorphic invariants. So we may assume that $M$ is either a projective $n$-manifold or a K\"{a}hler $3$-manifold. Then the conclusion follows from the celebrated works of \cite[Coro. 3]{BDPP} and \cite[Coro. 1.2]{Br} respectively.


\section{Proof of Theorem \ref{mainresult2}}\label{section5}
The idea of the proof is to conformally change $h$ to a new metric $\widetilde{h}$ such that the mean curvature form $\widetilde{K}$ of this new metric $\widetilde{h}$ is negative-definite and so Theorem \ref{mainresult2} follows from \cite[p. 52]{Ko}.

Set
$$f_0:=\gamma-\frac{\int_M\gamma\omega_0^n}{\int_M\omega^n_0}.$$
The assumptions in Theorem \ref{mainresult2} imply that
\begin{eqnarray}\label{condition}
\left\{ \begin{array}{ll}
f_0\in C^{\infty}(M;\mathbb{R}),\\
\gamma<f_0,\\
\int_Mf_0\cdot\omega^n_0=0.
\end{array} \right.
\end{eqnarray}

Then Lemma \ref{analogue of laplacian} tells us that there exists a $u_0\in C^{\infty}(M;\mathbb{R})$ such that \be\label{equation}\Delta_{c,\omega_0}(u_0)=-f_0.\ee

By (\ref{transform of mean curvature form}) the mean curvature form $\widetilde{K}$ of the new metric $\tilde{h}:=e^{u_0}h$ is exactly
$$(\widetilde{K}_{\alpha\bar{\beta}})=e^{u_0}\big[(K_{\alpha\bar{\beta}})
-f_0(h_{\alpha\bar{\beta}})\big],$$
which is negative-definite as $\gamma$ is the greatest eigenvalue of $(K_{\alpha\bar{\beta}})$ with respect to $(h_{\alpha\bar{\beta}})$ and $\gamma<f_0$.

\begin{remark}\label{mistake remark}
The mistake in the proof of \cite[p. 57, Thm 1.30]{Ko} is now clear. The author claimed the existence of $u_0$ as in (\ref{equation}) \big(\cite[p. 57, (1.32)]{Ko}\big) for \emph{any} Hermitian metric $\omega$ under the condition of $$\int_Mf_0\omega^n=0,$$ which is false due to the fact (\ref{Gauduchon integrate 0}).
\end{remark}

\section{Proof of Theorem \ref{mainresult3}}\label{section6}
Since
$$\Delta_{c,\omega}(f):=-\text{tr}_{\omega}
\sqrt{-1}\partial\bar{\partial}f=
<-\sqrt{-1}\partial\bar{\partial}f,\omega>_{\omega},$$
where $<>_{\omega}$ is the pointwise inner product with respect to $\omega$,
we have
\be\label{adjoint}
\begin{split}
\Delta^{\ast}_{c,\omega}(f)&=-\sqrt{-1}\partial^{\ast}
\bar{\partial}^{\ast}(f\omega)\\
&=-\sqrt{-1}(-\ast_{\omega}\cdot\bar{\partial}\cdot\ast_{\omega})
(-\ast_{\omega}\cdot\partial\cdot\ast_{\omega})(f\omega)\\
&=-\frac{\sqrt{-1}}{(n-1)!}
\ast_{\omega}\partial\bar{\partial}(f\omega^{n-1}),
\end{split}
\ee
where $\ast_{\omega}$ is the Hodge-star operator w.r.t. $\omega$ and the last equality is due to the facts that $\ast^2_{\omega}=-1$ on $(2n-1)$-forms and $$\ast_{\omega}(\omega)=\frac{1}{(n-1)!}\omega^{n-1}.$$

So the condition of $\partial\bar{\partial}(f\omega^{n-1})=0$ in Theorem \ref{mainresult3} is equivalent to $\Delta^{\ast}_{c,\omega}(f)=0$, which, together with Proposition \ref{Gauduchon properties}, yields the proof of the first part in Theorem \ref{mainresult3}.

Since $f\omega$ is a $(1,1)$-form, $d(f\omega)=0$ implies $\bar{\partial}(f\omega)=0$ and then $\partial\bar{\partial}(f^{n-1}\omega^{n-1})=0$. Thus the proof above implies that $f^{n-1}$ has constant sign and so is $f$.

\end{document}